\documentclass[12pt,sumlimits,intlimits]{amsart}
\usepackage{url,cite,hyperref,amsmath,amsthm,amssymb,mathtools}
\usepackage[margin=1.25in]{geometry}

\newtheorem{theorem}{Theorem}[section]

\newtheorem{lemma}[theorem]{Lemma}
\newtheorem{question}[theorem]{Question}
\newtheorem{definition}[theorem]{Definition}

\newcommand{\id}{\textup{id}}
\newcommand{\nd}{\textup{dim}_{\textup{nuc}}}
\newcommand{\la}{\langle}
\newcommand{\ra}{\rangle}

\newcommand{\Z}{\mathbb{Z}}

\newcommand{\C}{\mathbb{C}}
\newcommand{\N}{\mathbb{N}}

\newcommand{\T}{\mathbb{T}}

\newcommand{\mf}{\mathcal{F}}

\newcommand{\mh}{\mathcal{H}}

\numberwithin{equation}{section}

\title[Nuclear dimension of finitely-generated nilpotent group C*-algebras]{Finitely Generated Nilpotent group C*-algebras have finite nuclear dimension}
\address{Department of Mathematics, Miami University, Oxford, OH, 45056}
\author{Caleb Eckhardt}
\email{eckharc@miamioh.edu}
\author{Paul McKenney}
\email{mckennp2@miamioh.edu}
\thanks{C.E.\ was partially supported by NSF grant DMS-1262106.}

\date{}
\begin{document}
\maketitle
\begin{abstract}
We show that group C*-algebras of  finitely generated, nilpotent groups have finite nuclear dimension.
It then follows, from a string of deep results, that the C*-algebra $A$ generated by an irreducible representation of
such a group has decomposition rank at most 3.  If,  in addition, $A$ satisfies the universal coefficient theorem,
another string of deep results shows it is classifiable by its ordered K-theory and is approximately subhomogeneous.  We
observe that all C*-algebras generated by faithful irreducible representations of finitely generated, torsion free
nilpotent groups satisfy the universal coefficient theorem.
\end{abstract}
\section{Introduction}
The noncommutative dimension theories of Kirchberg and Winter (decomposition rank) and of Winter and Zacharias (nuclear
dimension) play a prominent role in the theory of nuclear C*-algebras.  This is especially apparent in Elliott's
classification program where finite noncommutative dimension is essential for a satisfying classification theory.  In
\cite{Winter10}, Winter and Zacharias express a hope that nuclear dimension will ``shed new light on the role of
dimension type conditions in other areas of noncommutative geometry." We share this hope and this work aims to use the
theory of nuclear dimension to shed new light on the representation theory of discrete nilpotent groups.  

A \emph{discrete} group is Type I (and therefore has a ``tractable" representation theory) if and only if it has an
abelian subgroup of finite index \cite{Thoma64}.  Therefore being Type I is a highly restrictive condition  for discrete
groups and therefore for most discrete groups, leaves many of the tools of classic representation theory out of reach.
Recent breakthroughs of several mathematicians (H. Lin, Z. Niu, H. Matui, Y. Sato and W. Winter to name a few) gave
birth to the possibility of classifying the C*-algebras generated by the irreducible representations of nilpotent groups
by their ordered K-theory. A key missing ingredient was knowing whether or not the group C*-algebras of finitely
generated nilpotent groups have finite nuclear dimension. Our main result (Theorem \ref{thm:mainthm}) supplies this
ingredient. In particular we show that the nuclear dimension of $C^*(G)$ is bounded by $10^{h(G)-1}\cdot h(G)!$ where
$h(G)$ is the Hirsch number of $G$ (see Section \ref{sec:ngfacts}). 

Each finitely generated nilpotent group has an algebraic ``basis" of sorts and the Hirsch number returns the size of
this basis--it is therefore not surprising to see $h(G)$ appear in the nuclear dimension estimate.

Fix a finitely generated nilpotent group $G$ and an irreducible representation $\pi$ of $G.$  The C*-algebra generated
by $\pi(G)$ is simple, nuclear, quasidiagonal with unique trace and, by a combination of Theorem \ref{thm:mainthm} with
many deep results (see Theorem \ref{thm:bigol} for a complete list), has finite decomposition rank.  Therefore if
$C^*(\pi(G))$ satisfies the universal coefficient theorem (see \cite{Rosenberg87}),  it is classified by its ordered
K-theory and isomorphic to an approximately subhomogeneous C*-algebra by \cite{Lin08,Matui12,Matui14a} (see
\cite[Corollary 6.2]{Matui14a}). We observe in Section \ref{sec:manicuredhands} that the work of  Rosenberg and Schochet
\cite{Rosenberg87} show that if $G$ is torsion free and $\pi$ is faithful (as a group homomorphism on $G$) then
$C^*(\pi(G))$ satisfies the universal coefficient theorem.

If $G$ is a two-step nilpotent group, it is well-known that the C*-algebras generated by irreducible representations of
$G$ are either finite dimensional or A$\T$-algebras.  Indeed Phillips showed in \cite{Phillips06} that all simple higher
dimensional non commutative tori (a class of C*-algebras that include $C^*(\pi(G))$ when $G$ is two-step and $\pi$ is an
irreducible, infinite dimensional representation) are A$\T$ algebras. In some sense this result forms the base case for
our induction proof (see below for a more detailed description).  Peeling back a couple layers, we mention that
Phillips' work relies on that of Elliott and Evans \cite{Elliott93} and Kishimoto \cite{Kishimoto98} (see also
\cite{Boca97} and \cite{LinQ96} for precursors to Phillips' Theorem).

Since in general a discrete nilpotent group $G$ is not Type I we are left with essentially no possibility of reasonably
classifying its irreducible representations up to unitary equivalence.  On the other hand if every primitive quotient of
$C^*(G)$ satisfies the UCT, then one could classify the C*-algebras generated by these representations by their ordered
K-theory.  This provides a dual viewpoint to the prevailing one of parametrizing irreducible representations by
primitive ideals of $C^*(G)$ or by the space of characters of $G$ (see \cite{Kaniuth06})--we thank Nate Brown for
sharing this nice observation with us.

Let us provide a broad outline of our proof. First we prefer to deal with torsion free groups.  Since every finitely
generated nilpotent group has a finite index torsion free subgroup we begin in Section \ref{sec:finext} by showing that
finite nuclear dimension is stable under finite extensions.  We then focus on the torsion free case. 

We proceed by induction on the Hirsch number (see Section \ref{sec:ngfacts}) of the nilpotent group $G.$  When dealing
with representation theoretic objects, (like a group C*-algebra) induction on the Hirsch number is sometimes more
helpful than induction on, say, the nilpotency class for the simple reason that non-trivial quotients of $G$ have Hirsch
number strictly less than $G$ while the nilpotency class of the quotient may be unchanged.

By \cite{Packer92} (see also Theorem \ref{thm:PackRae} below) we can view $C^*(G)$ as a continuous field over the dual
of its center, $Z(G)$.  Since $Z(G)$ is a finitely generated abelian group, our task is to bound the nuclear dimension
of the fibers as the base space is already controlled.

Since we are proceeding by induction we can more or less focus on those fibers (we call them $C^*(G,\widetilde{\gamma})$
) induced by characters $\gamma\in \widehat{Z(G)}$ that are faithful (as group homomorphisms) on $Z(G).$ In the case
that $G$ is a two step nilpotent group, then $C^*(G,\widetilde{\gamma})$ is a simple higher dimensional noncommutative
torus.  Phillips showed in \cite{Phillips06}  that any such C*-algebra is an A$\mathbb{T}$-algebra and therefore has
nuclear dimension (decomposition rank in fact) bounded by 1. Phillips' result is crucial for us as it allows us to
reduce to the case that the nilpotency class of $G$ is at least 3 and provides enough ``room" for the next step of the
proof.

We then find a subgroup $N$ of $G$, with strictly smaller Hirsch number, such that $G\cong N\rtimes\Z$.  ($N$ is simply
the subgroup generated by $Z_{n-1}(G)$ and all but one of the generators of $G / Z_{n-1}(G)$; see
section~\ref{sec:prelim} below for definitions.)  By induction we know $C^*(N)$ has finite nuclear dimension and so we
analyze the action of $\Z$ on fibers of $C^*(N).$  It turns out that there are two cases: In the first case, the fiber
is simple and the action of $\Z$ on the fiber is strongly outer and hence  the crossed product (which is a fiber of
$C^*(G)$) absorbs the Jiang-Su algebra by \cite{Matui14}. This in turn implies finite decomposition rank of the fiber
by a string of deep results (see Theorem \ref{thm:bigol}).  If the fiber is not simple, we can no longer employ the
results of \cite{Matui14},  but the non-simplicity forces the action restricted to the center of the fiber to have
finite Rokhlin dimension and so the crossed product has finite nuclear dimension by \cite{Hirshberg12}. 

\section{Preliminaries} \label{sec:prelim}

We assume the reader is familiar with the basics of group C*-algebras and discrete crossed products and refer them to
Brown and Ozawa \cite{Brown08} for more information.
\subsection{Facts about Nilpotent Groups} \label{sec:bignil}
\subsubsection{Group Theoretic Facts} \label{sec:ngfacts} We refer the reader to Segal's book \cite{Segal83} for more
information on polycyclic and nilpotent groups.  Here we collect some facts about nilpotent groups that we will use
frequently.
\\\\
Let $G$ be a group and define $Z_1(G)=Z(G)$ to be the center of $G.$  Recursively define $Z_n(G)\leq G$ to satisfy
$Z_n(G)/Z_{n-1}(G):=Z(G/Z_{n-1}(G)).$  A group $G$ is called \textbf{nilpotent} if $Z_n(G)=G$ for some $n$ and is called
\textbf{nilpotent of class \emph{n}} if $n$ is the least integer satisfying $Z_n(G)=G.$

A group $G$ is \textbf{polycyclic} if it has a normal series 
\begin{equation*}
\{ e \}\trianglelefteq G_1 \trianglelefteq \cdots \trianglelefteq G_{n-1}\trianglelefteq G_n
\end{equation*}
such that each quotient $G_{i+1}/G_i$ is cyclic. The number of times that $G_{i+1}/G_i$ is infinite is called the
\textbf{Hirsch number} of the group $G$ and is denoted by $h(G).$ The Hirsch number is an invariant of the group. If $G$
is polycyclic and $N$ is a normal subgroup then both $N$ and $G/N$ are polycyclic with
\begin{equation*}
h(G)=h(N)+h(G/N)
\end{equation*}
Every finitely generated nilpotent group is polycyclic.
\\\\
Let now $G$ be finitely generated and nilpotent.  Let $G_f$ denote the subgroup consisting of those elements with finite
conjugacy class.  By \cite{Baer48}, the center $Z(G)$ has finite index in $G_f.$ If in addition $G$ is torsion free,
then by \cite{Malcev49} the quotient groups $G/Z_{i}(G)$ are also all torsion free. These facts  combine to show that if
$G$ is torsion free, then every non-central conjugacy class is infinite.

\subsubsection{Representation Theoretic Facts} \label{sec:nilfacts}
 
\begin{definition}
Let $G$ be a group and $\phi:G\rightarrow\C$  a positive definite function with $\phi(e)=1.$  If $\phi$ is constant on
conjugacy classes, then we say $\phi$ is a \textbf{trace} on $G.$  It is clear that any such $\phi$ defines a tracial
state on $C^*(G).$  Following the representation theory literature we call an extreme trace a \textbf{character}.  Every
character gives rise to a factor representation and is therefore multiplicative on $Z(G).$
\end{definition}

Let $G$ be a finitely generated nilpotent group. Let $J$ be a primitive ideal of $C^*(G)$ (i.e. the kernel of an
irreducible representation).  Moore and Rosenberg showed in \cite{Moore76} that $J$ is actually a maximal ideal. Shortly
after this result,  Howe showed in \cite{Howe77} (see especially the introduction of \cite{Carey84}) that every
primitive ideal is induced from a \emph{unique}  character, i.e. there is a unique character $\phi$ on $G$ such that
\begin{equation*}
J=\{ x\in C^*(G):\phi(x^*x) = 0  \}.
\end{equation*}
Let $G$ be a finitely generated nilpotent group and $\phi$ a character on $G.$  Set $K(\phi)=\{ g\in G:\phi(g)=1  \}.$
By \cite{Howe77} and \cite{Carey84}, $G$ is \textbf{centrally inductive}, i.e. $\phi$ vanishes on the complement of $\{
g\in G: gK(\phi)\in (G/K(\phi))_f  \}.$  In this light we make the following definition.
\begin{definition} \label{def:trivial-extension}
  Let $N\le G$ and $\phi$ be a positive definite function on $N$.  We denote by $\widetilde{\phi}$ the \textbf{trivial
  extension} of $\phi$ to $G$, i.e. that extension of $\phi$ satisfying $\widetilde{\phi}(x) = 0$ for all $x\not\in N$.
\end{definition}
We use the following well-known fact repeatedly,
\begin{lemma} \label{lem:centervanish}
Let $G$ be a finitely generated torsion free nilpotent group and $\gamma$ a faithful character on $Z(G).$  Then
$\widetilde{\gamma}$ is a character of $G$.
\end{lemma}
\begin{proof}
Since $\gamma$ is a character and the extension $\widetilde{\gamma}$ is a trace on $G$, by a standard extreme point
argument there is a character $\omega$ on $G$ that extends $\gamma.$  Since $G$ is nilpotent, every non-trivial normal
subgroup intersects the center non-trivially. Since $K(\widetilde{\gamma})\cap Z(G)$ is trivial, we have
$K(\widetilde{\gamma})$ is also trivial. By the preceding section, $G_f=Z(G)$, i.e. $\omega$ vanishes off of $Z(G).$
But this means precisely that $\omega=\widetilde{\gamma}$.
\end{proof}

It is clear from the definitions that every nilpotent group is amenable and therefore by \cite{Lance73}, group
C*-algebras of nilpotent groups are nuclear.  In summary, for every primitive ideal $J$ of $C^*(G)$, the quotient
$C^*(G)/J$ is simple and nuclear with a unique trace.

\subsection{C*-facts}
\begin{definition}[\cite{Matui14}] \label{def:souter}
Let $G$ be a group acting on a C*-algebra $A$ with unique trace $\tau.$  Since $\tau$ is unique the action of $G$ leaves
$\tau$ invariant and therefore extends to an action on $\pi_\tau(A)''$, the von Neumann algebra generated by the GNS
representation of $\tau.$ If for each $g\in G\setminus\{ e \}$ the automorphism of $\pi_\tau(A)''$ corresponding to $g$
is an outer automorphism, then we say that the action is \textbf{strongly outer.}
\end{definition} 
The above definition can be modified to make sense even if $A$ does not have a unique trace (see \cite{Matui14}), but
we are only concerned with the unique trace case. The following theorem provides a key step in our main result.
\begin{theorem} [\textup{\cite[Corollary 4.11]{Matui14}}] \label{thm:souter}
Let $G$ be a discrete elementary amenable group acting on a unital, separable, simple, nuclear C*-algebra $A$ with property (SI) and finitely many extremal traces. If the action of
$G$ is strongly outer,  then the crossed product $A\rtimes G$ is $\mathcal{Z}$-absorbing.
\end{theorem}
The above theorem is actually given in greater generality in \cite{Matui14}.  We will only make use of it in the case
where $G$ is Abelian.  Another key idea for us is the fact that discrete amenable group C*-algebras decompose as
continuous fields over their centers.  First a definition, 
\begin{definition} \label{def:GNS}
  Let $G$ be a group and $\phi$ a positive definite function on $G$.  Then, we write $C^*(G,\phi)$ for the C*-algebra
  generated by the GNS representation of $\phi$.
\end{definition}
We recall the following special case of \cite[Theorem 1.2]{Packer92}:
\begin{theorem}  \label{thm:PackRae}
  Let $G$ be a discrete amenable group.  Then $C^*(G)$ is a continuous field of C*-algebras over  $\widehat{Z(G)},$ the
  dual group of $Z(G).$  Moreover, for each multiplicative character $\gamma\in \widehat{Z(G)}$ the fiber at $\gamma$ is
  isomorphic to $C^*(G,\widetilde{\gamma}).$ 
\end{theorem}
\begin{definition}
  Let $A$ be a C*-algebra. Denote by $\nd(A)$ the \textbf{nuclear dimension} of $A$ (see \textup{\cite{Winter10}} for
  the definition of nuclear dimension).
\end{definition}
It will be crucial for us to view $C^*(G)$ as a continuous field for our inductive step to work in the proof of our main
theorem.  The following allows us to control the nuclear dimension of continuous fields.
\begin{theorem}[\textup{\cite[Lemma 3.1]{Carrion11}, \cite[Lemma 5.1]{Tikuisis14}}] \label{thm:jose}
Let $A$ be a continuous field of C*-algebras over the finite dimensional compact space $X.$  For each $x\in X$, let
$A_x$ denote the fiber of $A$ at $x.$  Then
\begin{equation*}
  \nd(A)\leq (\textup{dim}(X)+1)(\sup_{x\in X}\nd(A_x)+1)-1.
\end{equation*}
\end{theorem}
Throughout this paper we never explicitly work with decomposition rank, nuclear dimension, property (SI) or the Jiang-Su
algebra $\mathcal{Z}$ (for example, we never actually build any approximating maps in proving finite nuclear dimension).
For this reason we do not recall the lengthy definitions of these properties but simply refer the reader to
\cite[Definition 3.1]{Kirchberg04} for decomposition rank, \cite[Definition 2.1]{Winter10} for nuclear dimension,
\cite{Jiang99,Elliott08} for the Jiang-Su algebra and its role in the classification program, and \cite[Definition
4.1]{Matui12} for property (SI).

Finally, for easy reference we gather together several results into
\begin{theorem}[R\o rdam, Winter, Matui and Sato] \label{thm:bigol}
  Let $A$ be a unital, separable, simple, nuclear, quasidiagonal C*-algebra with a unique tracial state.  If $A$ has any
  of the following properties then it has all of them.
\begin{enumerate}
  \item[(i)] Finite nuclear dimension.
  \item[(ii)] $\mathcal{Z}$-stability.
  \item[(iii)] Strict comparison.
  \item[(iv)] Property (SI) of Matui and Sato.
  \item[(v)] Decomposition rank at most 3.
\end{enumerate}
In particular, if $A$ is a primitive quotient of a finitely generated nilpotent group C*-algebra then it satisfies the
hypotheses of this Theorem.
\end{theorem}
\begin{proof}
  Winter showed (i) implies (ii) in \cite{Winter12}.  Results of Matui and Sato \cite{Matui12} and R\o rdam
  \cite{Rordam04} show that (ii), (iii) and (iv) are all equivalent.  Strict comparison and Matui and Sato's
  \cite{Matui14a} shows (v).  Finally (v) to (i) follows trivially from the definitions.
  If $G$ is a nilpotent group, by \cite{Eckhardt14}, any (primitive) quotient of $C^*(G)$ is quasidiagonal and therefore
  satisfies the hypotheses of the theorem by the discussion in Section \ref{sec:nilfacts}. 
\end{proof}
\section{Stability under finite extensions} \label{sec:finext}
In this section we show that if a nilpotent group $G$ has a finite index normal subgroup $H$ such that
$\nd(C^*(H))<\infty$, then $\nd(C^*(G))<\infty.$ Perhaps surprisingly, this portion of the proof is the most involved
and relies on several deep results of C*-algebra theory. Moreover we lean heavily on the assumption that $G$ is nilpotent.

This section exists because every finitely generated nilpotent group has a finite index subgroup
that is torsion free.  Reducing to this case gives the reader  a clear idea of what is happening without getting
bogged down in torsion.  We begin with the following special case that isolates most of the technical details.
\begin{theorem} \label{thm:findexp} Let $G$ be a finitely generated nilpotent group.  Suppose $H$ is a  normal subgroup
of finite index such that every primitive quotient of $C^*(H)$ has finite nuclear dimension.  Then every primitive
quotient of $C^*(G)$ has decomposition rank at most 3.
\end{theorem}
\begin{proof} 
We proceed by induction on $|G/H|.$  Since $G$ is nilpotent, so is $G/H.$ In particular $G/H$ has a cyclic group of
prime order as a quotient.  By our induction hypothesis we may therefore suppose that $G/H$ is cyclic of prime order
$p.$

Let $e, x,x^2,...,x^{p-1}\in G$ be coset representatives of $G/H.$ Let $\alpha$ denote the action of $G$ on
$\ell^\infty(G/H)$ by left translation.  It is well-known, and easy to prove, that $\ell^\infty(G/H)\rtimes_\alpha
G\cong M_p\otimes C^*(H)$ and that under this inclusion $C^*(H)\subseteq  C^*(G)\subseteq M_p\otimes C^*(H)$ we may
realize this copy of $C^*(H)$ as the C*-algebra generated by the diagonal matrices 
\begin{equation}
(\lambda_h,\lambda_{xhx^{-1}},...,\lambda_{x^{p-1}hx^{-(p-1)}}), \textrm{ with }h\in H. \label{eq:diagpic}
\end{equation}
Let $(\pi,\mathcal{H}_\pi)$ be an irreducible representation of $G.$ We show that $C^*(G)/\textup{ker}(\pi)$ has
decomposition rank less than or equal to 3. Let $\tau$ be the unique character on $G$ that induces $\textup{ker}(\pi)$
(see Section \ref{sec:nilfacts}) and set $K(\tau)=\{ x\in G: \tau(x)=1 \}.$  By assumption every primitive quotient of
$H/(H\cap K(\tau))$ has finite nuclear dimension. Since  $\Big(G/K(\tau)\Big)/\Big(H/(H\cap K(\tau))\Big)$ is a quotient of
$G/H$ it is either trivial, in which case we are done by assumption or it is isomorphic to $G/H.$ We may therefore
assume, without loss of generality, that
\begin{equation}
\textup{the character }\tau\textup{ that induces  }\textup{ker}(\pi)\textup{ is faithful on }Z(G). \label{eq:centerfaith}
\end{equation}

 By a well-known application of Stinespring's Theorem (see \cite[Theorem 5.5.1]{Murphy90}) there is an irreducible
 representation $\id_p\otimes \sigma$ of $M_p\otimes C^*(H),$ such that $\mh_\pi\subseteq \mh_{\id_p\otimes \sigma}$ and
 if $P:\mh_{\id_p\otimes \sigma}\rightarrow \mh_\pi$ is the orthogonal projection, then
\begin{equation}
P((\id_p\otimes \sigma)(x))P=\pi(x)\quad \textrm{ for all }\quad x\in C^*(G). \label{eq:extendrep}
\end{equation}
Let $J=\textup{ker}(\sigma)\subseteq C^*(H)\subseteq C^*(G)$ and $J_G\subseteq C^*(G)$ be the ideal of $C^*(G)$
generated by $J.$  By (\ref{eq:extendrep}) we have $J_G\subseteq \textup{ker}(\pi).$
\\\\
We now consider two cases:
\\\\
\textbf{Case 1:} $J_G$ is a maximal ideal of $C^*(G)$, i.e. $J_G=\textup{ker}(\pi).$
\\\\
In this case we have
\begin{equation*}
C^*(H)/J\subseteq C^*(G)/J_G\subseteq M_p\otimes (C^*(H)/J).
\end{equation*}
We would like to reiterate that in general the copy of $C^*(H)/J$, is \emph{not} conjugate to the diagonal copy
$1_p\otimes (C^*(H)/J)$, but rather the twisted copy of (\ref{eq:diagpic}).  If $C^*(H)/J$ is conjugate to the diagonal
copy, then the proof is quite short (see the last paragraph of Case 1b), so most of the present proof consists of
overcoming this difficulty.
\\\\
Let $\Z_p\cong G/H$ denote the cyclic group of order $p.$ Define an action $\beta$ of $\Z_p$ on
$\ell^\infty(G/H)\rtimes_\alpha G$ by $\beta_t(f)(s)=f(s-t)$ for all $f\in \ell^\infty(G/H)$ (i.e. $\beta$ acts by left
translation on $\ell^\infty(G/H)$) and $\beta(\lambda_g)=\lambda_g$ for all $g\in G.$  Since the $G$-action $\alpha$ and
$\Z_p$-action $\beta$ commute with each other it is easy to see that $\beta$ defines an action of $\Z_p$ on
$\ell^\infty(G/H)\rtimes_\alpha G.$ Moreover 
\begin{equation*}
\beta_t(x)=x\quad \textrm{ for all }t\in \Z/p\Z\textrm{ if and only if }x\in C^*(G).
\end{equation*}
In particular $\beta$ fixes $J_G\subseteq C^*(G)$ pointwise. Notice that $M_p\otimes J$ is generated by the
$\beta$-invariant set $\{e_1  a_1 + \cdots + e_p  a_p\;|\; a_i\in J\}$, where $e_i$ denotes the $i$th minimal projection
of $M_p$.  Therefore, $\beta$  leaves $M_p\otimes J$ invariant (although not pointwise), and induces an automorphism of
$M_p\otimes (C^*(H)/J).$ Moreover the fixed-point subalgebra of this induced automorphism is exactly $C^*(G)/J_G$.
We now split further into two subcases based on the behavior of $\beta.$
\\\\
\textbf{Case 1a:} The action $\beta \curvearrowright M_p\otimes (C^*(H)/J)$ is strongly outer (Definition \ref{def:souter}).
\\\\
By assumption,   $M_p\otimes (C^*(H)/J)$ has finite nuclear dimension. By Theorem \ref{thm:bigol}, $M_p\otimes
(C^*(H)/J)$ then has property (SI).  Since $\beta$ is strongly outer, by \cite[Corollary 4.11]{Matui14}, the crossed
product $M_p\otimes (C^*(H)/J)\rtimes_\beta \Z_p$ is $\mathcal{Z}$-stable.  Since $M_p\otimes (C^*(H)/J)$ has a unique
trace and $\ell^\infty(p)\rtimes \Z_p\cong M_p$, it follows that $[M_p\otimes (C^*(H)/J)]\rtimes_\beta \Z_p$ has unique
trace. It follows from \cite{Eckhardt14} that $[M_p\otimes (C^*(H)/J)]\rtimes_\beta \Z_p$ is quasidiagonal.  Therefore
by Theorem \ref{thm:bigol}, $M_p\otimes (C^*(H)/J)\rtimes_\beta \Z_p$ has decomposition rank at most 3.

By \cite{Rosenberg79}, the fixed point algebra of $\beta$, i.e. $C^*(G)/J_G$ is isomorphic to a corner of $M_p\otimes
(C^*(H)/J)\rtimes_\beta \Z_p$, which by  Brown's  isomorphism theorem \cite{Brown77} implies that $C^*(G)/J_G$ is
stably isomorphic to $M_p\otimes (C^*(H)/J)\rtimes_\beta \Z_p$ and therefore $C^*(G)/J_G=C^*(G)/\textup{ker}(\pi)$ has
decomposition rank at most 3 by \cite[Corollary 3.9]{Kirchberg04}. 
\\\\
\textbf{Case 1b:} The action $\beta \curvearrowright M_p\otimes (C^*(H)/J)$ is not stongly outer (Definition \ref{def:souter}).
\\\\
Choose a generator $t$ of $\Z_p$ and set $\beta=\beta_t$ (note that every $\beta_t$ is strongly outer or none of them are).
We will first show that $\beta$ is actually an inner automorphism of $M_p\otimes (C^*(H)/J).$

The unique trace on $M_p\otimes (C^*(H)/J)$ restricts to the unique trace $\tau$  on $C^*(G)/J_G.$
We will use the common letter $\tau$ for both of these traces.

Let $G_f\leq G$ be the subgroup consisting of those elements with finite conjugacy classes.  By \cite[Lemma 3]{Baer48},
$G_f/Z(G)$ is finite.  Let $(\pi_\tau,L^2)$ be the GNS representation of $\ell^\infty(G/H)\rtimes_\alpha G$ associated
with $\tau.$ 

Since $\tau$ is multiplicative on $Z(G)$ it easily follows that  
\begin{equation*}
\la \lambda_x,\lambda_y   \ra_\tau=\tau(y^{-1}x)\in \T, \quad \textrm{for }x,y\in Z(G).
\end{equation*}
By the Cauchy-Schwarz inequality it follows that in $L^2$ we have 
\begin{equation}
\lambda_x=_{L^2}\tau(y^{-1}x)\lambda_y \textup{ for all }x,y\in Z(G).\label{eq:Zequality}
\end{equation}
Let $x_1,...,x_n\in G_f$ be coset representatives of $G_f/Z(G)$ and let $C\subseteq G$ be a set of coset
representatives for $G/G_f.$

The preceding discussion shows that the following set  spans a dense subset of $L^2:$
\begin{equation}
\{f\lambda_{tx_i}: f\in\ell^\infty(G/H), t\in C, 1\leq i\leq n\}. \label{eq:L2dense}
\end{equation}
Since $\tau$ is $\beta$-invariant, for each minimal projection $e\in \ell^\infty(p)$ and $t\in G$ we have
$\tau(e\lambda_t)=\tau(\beta(e)\lambda_t).$ So for each $f\in \ell^\infty(G/H)$ and $t\in G$ we have
$\tau(f\lambda_t)=\tau(f)\tau(t)$.

Combining this with the fact that $\tau$ vanishes on infinite conjugacy classes, we have
\begin{equation}
\la f\lambda_t,g\lambda_s  \ra_\tau=0 \textup{ for all }f,g\in \ell^\infty(G/H), \textrm{ and }t,s\in C, t\neq s.
\label{eq:orthog}
\end{equation}
Since $\beta$ is not strongly outer there is a unitary $W\in \pi_\tau(\ell^\infty(G/H)\rtimes_\alpha G)''$ such that 
\begin{equation*}
W\pi_\tau(x)W^*=\pi_\tau(\beta(x))\quad \textrm{ for all }x\in \ell^\infty(G/H)\rtimes_\alpha G.
\end{equation*}
Since $\beta$ leaves $\pi_\tau(G)$ pointwise invariant, $W$ must commute with $\pi_\tau(\lambda_t)$ for all $t\in G.$
For $t\in G$, let $\textup{Conj}(t)=\{ sts^{-1}:s\in G  \}$ be the conjugacy class of $t.$  
\\\\
Let $s\in G\setminus G_f.$  Suppose first that $\textup{Conj}(s)$ intersects infinitely many $G/G_f$ cosets.  Let
$(s_n)_{n=1}^\infty$ be a sequence from $G$ such that the cosets  $s_nss_n^{-1}G_f$ are all distinct. Let $f\in
\ell^\infty(G/H).$  Since $W$ commutes with the $\lambda_t$'s, for each $n\in \mathbb{N}$ we have 
\begin{align*}
  \la  W,f\lambda_s  \ra_\tau & =\la \lambda_{s_n^{-1}}W\lambda_{s_n},f\lambda_s  \ra_\tau\\
      &=\la W,\alpha_{s_n}(f)\lambda_{s_nss_n^{-1}}  \ra_\tau.
\end{align*}
By (\ref{eq:orthog}), the vectors $\{ \alpha_{s_n}(f)\lambda_{s_nss_n^{-1}}:n\in \N \}$ form an orthogonal family of
vectors, each with the same $L^2$ norm.  Since $W$ has $L^2$ norm equal to 1, this implies that  $\la W,f\lambda_s
\ra=0$ for all $f\in \ell^\infty(G/H).$
\\\\
Suppose now that $\textup{Conj}(s)$ intersects only finitely many $G/G_f$ cosets. Since $\textup{Conj}(s)$ is infinite,
and $Z(G)$ has finite index in $G_f$, there is a $y\in G$ such that $\textup{Conj}(s)\cap yZ(G)$ is infinite. Let $s_n$
be a  sequence from $G$ and $t_n$ a sequence of distinct elements from $Z(G)$ so $s_nss_n^{-1}=yt_n$ for all $n\in
\mathbb{N}.$ 

Let $f\in \ell^\infty(G/H).$  Since the set $\{  \alpha_g(f):g\in G \}$ is finite we may, without loss of generality,
suppose that $\alpha_{s_n}(f)=\alpha_{s_m}(f)$ for all $n,m\in\mathbb{N}.$ By (\ref{eq:Zequality}), we have
$\lambda_{yt_n}=_{L^2}\tau(t_m^{-1}t_n)\lambda_{yt_m}$ for all $n,m\in \mathbb{N}.$ We then have
\begin{align*}
  \la  W,f\lambda_s  \ra_\tau & =\la \lambda_{s_n^{-1}}W\lambda_{s_n},f\lambda_s  \ra_\tau\\
    &=\la W,\alpha_{s_n}(f)\lambda_{s_nss_n^{-1}}  \ra_\tau\\
    &=\la W,\alpha_{s_1}(f)\lambda_{yt_n}\ra_\tau\\
    &=\tau(t_1^{-1}t_n)\la W,\alpha_{s_1}(f)\lambda_{yt_1}\ra_\tau\\
\end{align*}
In particular, we have
\begin{equation*}
  \tau(t_1^{-1}t_1)\la W,\alpha_{s_1}(f)\lambda_{yt_1}\ra_\tau=\tau(t_1^{-1}t_2)\la
  W,\alpha_{s_1}(f)\lambda_{yt_1}\ra_\tau.
\end{equation*}
Since $\tau$ is faithful on $Z(G)$ (by (\ref{eq:centerfaith})) we have $1=|\tau(t_1^{-1}t_2)|,$ and
$\tau(t_1^{-1}t_2)\neq \tau(t_1^{-1}t_1)=1.$  Hence $\la  W,f\lambda_s  \ra_\tau=0.$
\\\\
We have shown that for all $s\in G\setminus G_f$ and $f\in \ell^\infty(G/H)$ we have $\la W,f\lambda_s \ra_\tau=0.$  By
(\ref{eq:L2dense}) it follows that
\begin{equation*}
  W\in \textup{span}\{  \pi_\tau(f\lambda_{x_i}): f\in \ell^\infty(G/H), 1\leq i\leq n \}\subseteq
  \pi_\tau(\ell^\infty(G/H)\rtimes G)\cong M_p\otimes (C^*(H)/J)
\end{equation*}
By the way that $\beta$ acts on $\ell^\infty(G/H)$ there are $a_1,...,a_p\in C^*(H)/J$ so
\begin{equation*}
W=e_{1p}\otimes a_1+\sum_{i=2}^pe_{i,i-1}\otimes a_i.
\end{equation*}
Set $U=\textup{diag}(1,a_2^*,a_2^*a_3^*,...,a_2^*a_3^*\cdots a_p^*)\in M_p\otimes (C^*(H)/J).$  Since $W$ commutes with
$C^*(H)/J$, by (\ref{eq:diagpic})  it follows that
\begin{equation}
  U \Big(C^*(H)/J\Big)  U^*=1_{M_p}\otimes C^*(H)/J\subseteq U( C^*(G)/J_G ) U^*\subseteq M_p\otimes (C^*(H)/J).
  \label{eq:twistCH}
\end{equation}
Let $\mathcal{U}$ be a free ultrafilter on $\mathbb{N}$ and for a C*-algebra $A$, let $A^\mathcal{U}$ denote the
ultrapower of $A.$ We think of $A\subset A^\mathcal{U}$ via the diagonal embedding and write $A'\cap A^\mathcal{U}$ for
those elements of the ultrapower that commute with this diagonal embedding.

By Theorem \ref{thm:bigol}, $C^*(H)/J$ is $\mathcal{Z}$-stable. By \cite[Theorem 7.2.2]{Rordam02}, there is an embedding
$\phi$ of $\mathcal{Z}$ into $(C^*(H)/J)^\mathcal{U}\cap (C^*(H)/J)'.$  By both inclusions of (\ref{eq:twistCH}) it is
clear that $1_{M_p}\otimes \phi$ defines an embedding of $\mathcal{Z}$ into  $(C^*(G)/J_G)^\mathcal{U}\cap
(C^*(G)/J_G)'.$  Therefore, again by \cite[Theorem 7.2.2]{Rordam02} it follows that $C^*(G)/J_G$ is
$\mathcal{Z}$-stable.  By Theorem \ref{thm:bigol}, $C^*(G)/J_G$ has decomposition rank at most 3.
\\\\
\textbf{Case 2:} The ideal $J_G$ is not maximal.
Recall the definition of $\sigma$ from the beginning of the proof.
\\\\
For each $i=0,...,p-1$ define the representations of $H$,  $\sigma_i(h)=\sigma(x^i h x^{-i}).$  By \cite[Section
3]{Eckhardt14} either all of $\sigma_i$ are unitarily equivalent to each other or none of them are.  We treat these
cases separately.
\\\\
\textbf{Case 2a:} All of the $\sigma_i$ are unitarily equivalent to each other.
\\\\
By the proof of Lemma 3.4 in \cite{Eckhardt14} it follows that there is a unitary $U\in M_p\otimes C^*(H)/J$ such that
$U(C^*(H)/J)U^*=1_p\otimes (C^*(H)/J).$ Moreover from the same proof there is a projection $q\in M_p\otimes
1_{C^*(H)/J}$ that commutes with $U(\id_{M_p}\otimes \sigma(C^*(G)))U^*$ so $q(U[\id_{M_p}\otimes
\sigma(C^*(G))]U^*)\cong C^*(G)/\textup{ker}(\pi).$ We therefore have a chain of inclusions
\begin{equation*}
q\otimes C^*(H)/J\subseteq q(U[\id_{M_p}\otimes \sigma(C^*(G))]U^*)\subseteq qM_pq\otimes C^*(H)/J.
\end{equation*}
We can now complete the proof as in the end of Case 1b (following nearly verbatim everything that follows
(\ref{eq:twistCH})).
\\\\
\textbf{Case 2b.} None of the $\sigma_i$ are unitarily equivalent to each other. 
\\\\
From the proof of Lemma 3.5 in \cite{Eckhardt14} there is a projection $q\in \ell^\infty(p)\otimes 1_{C^*(H)/J}$ that
commutes with $(\id_{M_p}\otimes \sigma)(C^*(G))$ so $q(\id_{M_p}\otimes \sigma(C^*(G)))\cong C^*(G)/\textup{ker}(\pi).$
But since $G$ acts transitively on $G/H$ (and hence ergodically on $\ell^\infty(G/H)$) we must have $q=1.$  But this
implies that $\textup{ker}(\id_{M_p}\otimes \sigma|_{C^*(G)})=\textup{ker}(\pi).$ 

Recall the coset representatives $e,x,x^2,...,x^{p-1}$ of $G/H.$ Each $x\in C^*(G)$ can be written uniquely as
$\sum_{i=0}^{p-1}a_i\lambda_{x^i}$ for some $a_i\in C^*(H).$ Fix an $0\leq i\leq p-1$ and consider $\lambda_{x^i}\in
M_p\otimes C^*(H).$ If there is an index $(k,\ell)$ such that the $(k,\ell)$-entry of $\lambda_{x^i}$ is non-zero , then
for any $j\neq i$ the $(k,\ell)$-entry of $\lambda_{x^j}$ must be 0.   From this observation it follows that 
\begin{equation*}
  \id_p\otimes\sigma\Big(\sum_{i=0}^{p-1}a_i\lambda_{x^i}\Big)=0, \textrm{ if and only if }\quad
  (\id_p\otimes\sigma)(a_i)=0 \quad \textrm{ for all }1\leq i\leq p
\end{equation*}
In other words $\textup{ker}(\pi)=\textup{ker}(\id_{M_p}\otimes \sigma|_{C^*(G)})\subseteq J_G$
and we are done by Case 1. 

\end{proof}
\begin{lemma} \label{lem:thoma}
Let $G$ be a finitely generated nilpotent group and $N$ a finite index subgroup of $Z(G)$ and $\phi$ a faithful
multiplicative character on $N.$  Then there is a finite set $\mf$ of characters of $G$ such that $\widetilde{\phi}$
(see Definition \ref{def:GNS}) is in the convex hull of $\mf.$
\end{lemma}
\begin{proof}
This result follows from Thoma's work on characters \cite{Thoma64} (see also the discussion on page 355 of
\cite{Kaniuth06}).  For the convenience of the reader we outline a proof and keep the notation of \cite{Kaniuth06}.  Let
$G_f\leq G$ denote the group consisting of elements with finite conjugacy class. By \cite{Baer48}  $Z(G)$ (and hence
$N$) has finite index in $G_f.$

Let $\pi$ be the GNS representation of $G_f$ associated with $\widetilde{\phi}|_{G_f}.$  Since $N$ has finite index in
$G_f$ and $\pi(N)\subseteq \C$ it follows that $\pi(G_f)$ generates a finite dimensional C*-algebra. Therefore there are
finitely many characters $\omega_1,...,\omega_n$ of $G_f$ that extend $\phi$ (trivially as $\pi(N)\subseteq \C$) and a
sequence of positive scalars $\lambda_i$ such that  
\begin{equation}
\widetilde{\phi}|_{G_f}=\sum_{i=1}^n\lambda_i\omega_i \label{eq:convcomb}
\end{equation}
The positive definite functions $\widetilde{\omega}_i$ on $G$ need not be traces, but this is easily remedied  by the
following averaging process (which we took from \cite{Kaniuth06} and \cite{Thoma64}).

Let $x\in G_f.$  Then the centralizer of $x$ in $G$, denoted $C_G(x)$, has finite index.  Let $A_x$ be a complete set
of coset representatives of $G/C_G(x).$  For each $1\leq i\leq n$ and $x\in G_f$ define
\begin{equation}
  \widetilde{\omega}_i^G(x)=\frac{1}{[G:C_G(x)]}\sum_{a\in A_x}\widetilde{\omega}_i(axa^{-1}). \label{eq:inducedtrace}
\end{equation}
Then each $\widetilde{\omega}_i^G$ is extreme in the space of $G$-invariant traces on $G_f$ (see \cite[Page
355]{Kaniuth06}).  Since each $\widetilde{\omega}_i^G$ is $G$-invariant, the trivial extension to $G$ (which we still
denote by $\widetilde{\omega}_i^G$) is a trace on $G.$ We show that the $\widetilde{\omega}_i^G$ are actually characters
on $G.$

By a standard convexity argument, there is a character $\omega$ on $G$ which extends $\widetilde{\omega_i}^G$.  Let
$K(\omega)=\{ x\in G: \omega(x)=1 \}.$  Then $K(\omega)\subseteq G_f.$  Indeed if there is a $g \in K(\omega)\setminus
G_f$, then $g$ necessarily has infinite order (otherwise the torsion subgroup of $G$ would be infinite).  Since every
finitely generated nilpotent group has a finite index torsion free subgroup and every nontrivial subgroup of a nilpotent
group intersects the center non-trivially,  this would force $K(\omega)\cap Z(G)$ to have non-zero Hirsch number, but by
assumption $\phi=\omega|_N$ is faithful on a finite index subgroup of $Z(G).$ 

Since $G$ is finitely generated and nilpotent, it is centrally inductive (see \cite{Carey84}).  This means that
$\omega$ vanishes outside of $G_f(\omega)=\{ x\in G: xK(\omega)\in (G/K(\omega))_f  \}.$  But since $K(\omega)\subseteq
G_f$ and $\omega|_N=\phi$, it follows that $K(\omega)$ is finite.  From this it follows that $G_f(\omega)=G_f,$ i.e.
$\omega$ must vanish outside of $G_f.$ But this means precisely that $\omega=\widetilde{\omega}_i^G.$  By
(\ref{eq:convcomb}) and (\ref{eq:inducedtrace}) for $x\in G_f$ we have
\begin{align*}
  \sum_{i=1}^n\lambda_i\widetilde{\omega}_i^G(x)&=\sum_{i=1}^n\lambda_i\Big( \frac{1}{[G:C_G(x)]}\sum_{a\in A_x}\widetilde{\omega}_i(axa^{-1})    \Big)\\
    &=\frac{1}{[G:C_G(x)]}\sum_{a\in A_x}\Big( \sum_{i=1}^n\lambda_i \widetilde{\omega}_i(axa^{-1})  \Big)\\
    &=\frac{1}{[G:C_G(x)]}\sum_{a\in A_x}\widetilde{\phi}(x)\\
    &= \widetilde{\phi}(x).
\end{align*}
For $x\not\in G_f$ we clearly have $\sum_{i=1}^n\lambda_i\widetilde{\omega}_i^G(x)=0=\widetilde{\phi}(x).$

\end{proof}

\begin{lemma} \label{lem:findex}
Set $f(n)=10^{n-1}n!.$ Let $G$ be a finitely generated nilpotent group. Let $H\unlhd G$ be normal of finite index.  If
$\nd(C^*(H/N))\leq f(h(H/N))$ for every normal subgroup of $H$, then $\nd(C^*(G/K))\leq f(h(G/K))$ for every normal
subgroup of $G.$
\end{lemma}
\begin{proof}
We proceed by induction on $h(G).$  If $h(G)=0$, then $G$ is finite and there is nothing to prove.  So assume that for
every finitely generated nilpotent group $A$ with $h(A)<h(G)$ that satisfies $\nd(C^*(A/N))\leq f(h(A/N))$ for every
normal subgroup $N$ of $A$, we have $\nd(C^*(A'/N'))\leq f(h(A'/N'))$ where $A'$ is a finite normal extension of $A$ and $N'$
is a normal subgroup of $A'.$

Let now $H$ be a finite index normal subgroup of $G$ that satisfies the hypotheses. If $G/Z(G)$ is finite, then
$\nd(C^*(G))=h(G)\leq f(h(G))$ by Theorem \ref{thm:jose}.  Suppose that $G/Z(G)$ is infinite, i.e. $h(Z(G))<h(G).$ Since
for any quotient $G/K$ of $G$, the group $H/(H\cap K)$ has finite index in $G/K$ it suffices to show that
$\nd(C^*(G))\leq f(G).$

We use Theorem \ref{thm:PackRae} to view $C^*(G)$ as a continuous field over $\widehat{Z(G)}.$  We estimate the nuclear
dimension of the fibers. Let $\gamma\in \widehat{Z(G)}.$ Suppose first that $h(\textup{ker}(\gamma))>0.$  The fiber
$C^*(G, \widetilde{\gamma} )$ is a quotient of the group C*-algebra $C^*(G/\textup{ker}(\gamma)).$  By our induction
hypothesis,
\begin{equation*}
\nd(C^*(G/\textup{ker}(\gamma)))\leq f(h(G)-1).
\end{equation*}
Now suppose on the other hand that $F=\textup{ker}(\gamma)$ is finite. If $x\in G$ with $[x,y]\in F$ for all $y\in G$, then
$yx^{|F|}y^{-1}=x^{|F|}$ for all $y\in G,$ i.e. $x^{|F|}\in Z(G).$  Hence $Z(G/F)/ (Z(G)/F)$ is a finitely generated,
nilpotent torsion group, hence is finite. 

We therefore replace $G$ with $G/F$ and suppose that $Z(G)$ contains a finite index subgroup $N$ such that $\gamma$ is a
faithful homomorphism on $N$ (Notice we can not say that $\gamma$ is faithful on $Z(G)$ since it may not be the case
that $Z(G/F)=(Z(G)/F)$). By Lemma \ref{lem:thoma}, there are finitely many distinct characters $\omega_1,...,\omega_n$
on $G$ such that $\widetilde{\gamma}$ is a convex combination of the $\omega_i.$  The GNS representation associated with
$\widetilde{\gamma}$ is then the direct sum of the GNS representations associated with the $\omega_i.$ Let $J_i$ be the
maximal ideal of $C^*(G)$ induced by $\omega_i$ (Section \ref{sec:nilfacts}).  Since all of the $J_i$ are maximal and
distinct it follows (from purely algebraic considerations) that
\begin{equation*}
 C^*(G,\widetilde{\gamma})\cong \bigoplus_{i=1}^n C^*(G,\omega_i).
\end{equation*}
Since each $\omega_i$ is a character on $G$, it follows by Theorem \ref{thm:findexp} that the decomposition rank of
$C^*(G,\omega_i)$ is bounded by 3 which also bounds the decomposition rank of $C^*(G,\widetilde{\gamma})$ by 3. 

Since the nuclear dimension of all the fibers of $C^*(G)$ are bounded by $f(h(G)-1)$, by Theorem~\ref{thm:jose} we have
$\nd(C^*(G))\leq 2h(G) f(h(G)-1) \le f(h(G))$.
\end{proof}

\section{Main Result} \label{sec:manicuredhands}
The work of Matui and Sato on strongly outer actions (see Theorem \ref{thm:souter}) and of Hirshberg, Winter and
Zacharias \cite{Hirshberg12} on Rokhlin dimension are both crucial to the proof of our main result.  In our case, their
results turned  extremely difficult problems into ones with more or less straightforward solutions.
\begin{lemma} \label{lem:roots}
  Let $\alpha$ be an outer automorphism of a torsion free nilpotent group $G$.  Then for every $a\in G$, the following
  set is infinite.
  \begin{equation*}
    \{  s^{-1}a\alpha(s): s\in G \}.
  \end{equation*}
\end{lemma}
\begin{proof}
Suppose that for some $a$, the above set is finite.  Then for any $s\in G$ there are $0\leq m<n$ so that
\begin{equation*}
  s^{-m}a\alpha(s^m)=s^{-n}a\alpha(s^n)\quad \textrm{ or }\quad a^{-1}s^{n-m}a=\alpha(s^{n-m}).
\end{equation*}
Therefore
\begin{equation*}
  s^{n-m}=a\alpha(s^{n-m})a^{-1}=(a\alpha(s)a^{-1})^{n-m}.
\end{equation*}
Since $G$ is nilpotent and torsion free it has unique roots (see \cite{Malcev49} or \cite[Lemma 2.1]{Baumslag71}), i.e.
$s=a\alpha(s)a^{-1}$, or $\alpha$ is an inner automorphism.
\end{proof}
\begin{lemma} \label{lem:strongouter}
Let $G$ be a torsion free nilpotent group of class $n\geq3.$  Suppose that $G=\la N,x  \ra$ where $x\in G\setminus
Z_{n-1}(G)$, $N\cap \la x\ra=\{ e\}$ and $Z(G)=Z(N).$  Let $\phi$ be a trace on $G$ that is multiplicative on $Z(G)$ and
that vanishes off of $Z(G).$  Let $\alpha$ be the automorphism of $C^*(N,\phi|_N)$ induced by conjugation by $x.$  Then
$\alpha$ is strongly outer.
\end{lemma}
\begin{proof}
  We first show that $\alpha$ induces an outer automorphism of $N/Z(G).$  If not, then there is a $z\in N$ such that
\begin{equation*}
xax^{-1}Z(G)=z^{-1}azZ(G) \quad \textrm{ for all }a\in N.
\end{equation*}
In other words $zx\in Z_2(G).$  Since $x\not\in Z_{n-1}(G)$ we also have $z\not\in Z_{n-1}(G).$  Then
$z^{-1}Z_{n-1}(G)=xZ_{n-1}(G)$, from which it follows that $z\not\in N.$
\\\\
Since $G$ is torsion free so is $N.$  Since $\phi$ is multiplicative on $Z(N)$ and vanishes outside of $Z(N)$ it follows
that $\phi$ is a character (see Lemma \ref{lem:centervanish}) and therefore $C^*(N,\phi)$ has a unique trace (see
Section \ref{sec:bignil}).  Let $(\pi_\phi,L^2(N,\phi))$ be the GNS representation associated with $\phi.$  By way of
contradiction suppose there is a $W\in \pi_\phi(N)''$ such that $W\pi_\phi(g)W^*=\pi_\phi(\alpha(g))$ for all $g\in N.$

For each $s\in N$ we let $\delta_s\in L^2(N,\phi)$ be its canonical image.  Notice that if $aZ(N)\neq bZ(N)$, then
$\delta_a$ and $\delta_b$ are orthogonal. In particular for any complete choice of coset representatives $C\subseteq N$
for $N/Z(G)$, the set $\{ \delta_c:c\in C  \}$ is an orthonormal basis for $L^2(N,\phi).$ Since $W$ is in the weak
closure of the GNS representation $\pi_\phi$ it is in $L^2(N,\phi)$ with norm 1.  Therefore there is some $a\in N$ such
that $\la W, \delta_a\ra\neq 0.$

We now have, for all $s\in N$,
\begin{align*}
\la W,\delta_a  \ra &=\la W\pi_\phi(s),\delta_{as}  \ra\\
&=\la \pi_\phi(\alpha(s))W,\delta_{as}   \ra\\
&=\la  W,\delta_{\alpha(s)^{-1}as} \ra .
\end{align*}
Since $\alpha$ is outer on $N/Z(N)$, by Lemma \ref{lem:roots}  the set $\{  \alpha(s)^{-1}asZ(N):s\in N \}$ is infinite.
Since distinct cosets provide orthogonal vectors of norm 1, this contradicts the fact that $W\in L^2(N,\phi).$

Essentially the same argument shows that any power of $\alpha$ is also not inner on $\pi_\tau(N)''$, i.e. the action is
strongly outer.
\end{proof}
We refer the reader to the paper \cite{Hirshberg12} for information on Rokhlin dimension of actions on C*-algebras.  For
our purposes we do not even need to know what it is, simply that our actions have finite Rokhlin dimension. Therefore we
omit the somewhat lengthy definition \cite[Definition 2.3]{Hirshberg12}. We do mention the following corollary to the
definition of Rokhlin dimension:  If $\alpha$ is an automorphism of a C*-algebra $A$ and there is an $\alpha$-invariant,
unital subalgebra $B\subseteq Z(A)$ such that the action of $\alpha$ on $B$ has Rokhlin dimension bounded by $d$, then
the action of $\alpha$ on $A$ also has Rokhlin dimension bounded by $d.$

\begin{lemma} \label{lem:rokdim}
Let $G$ be a finitely generated, torsion free nilpotent group.  Suppose that $G=\la N,x \ra$ where $N\lhd G$, $N\cap \la
x \ra = \{e\}$, $Z(G)\subseteq Z(N)$, and $Z(G)\neq Z(N).$ Let $\phi$ be a trace on $G$ that is multiplicative on $Z(G)$
and vanishes off $Z(G).$  Let $\alpha$ be the automorphism of $C^*(N,\phi|_N)$ defined by conjugation by $x.$ Then
the Rokhlin dimension of $\alpha$ is 1.
\end{lemma}
\begin{proof}
Consider the action of $\alpha$ restricted to $Z(N).$  Since $Z(N)\neq Z(G)$, $\alpha$ is not the identity on $Z(N).$
Since $Z(N)$ is a free abelian group we have $\alpha\in GL(\Z,d)$ where $d$ is the rank of $Z(N).$ Since $G$ is
nilpotent, so is the group $Z(N)\rtimes_\alpha \Z.$  Therefore $(1-\alpha)^d=0.$   In particular there is a $y\in Z(N)$
such that $(1-\alpha)(y)\neq0$ but $(1-\alpha)^2(y)=0.$ From this we deduce that $\alpha(y)=y+z$ for some $z\in
Z(G)\setminus\{ 0 \}.$  

Therefore the action of $\alpha$ on $C^*(\pi_\phi(y))\cong C(\T)$ is a rotation by $\phi(z).$  Since $\phi$ is faithful
we have $\phi(z)=e^{2\pi i\theta}$ for some irrational $\theta.$  By \cite[Theorem 6.1]{Hirshberg12} irrational
rotations of the circle have Rokhlin dimension 1. Since $C^*(\pi_\phi(y))\subseteq Z(C^*(N,\phi))$, the remark preceding
this lemma shows that the Rokhlin dimension of $\alpha$ acting on $C^*(N,\phi)$ is also equal to 1.
\end{proof}
\begin{theorem} \label{thm:mainthm}
Define $f:\N\rightarrow\N$ by $f(n)=10^{n-1} n!.$ Let G be a finitely generated nilpotent group.  Then $\nd(C^*(G))\leq
f(h(G))$.
\end{theorem}
\begin{proof}
We proceed by induction on the Hirsch number of $G.$ If $h(G)=0$, there is nothing to prove.  It is well-known that $G$
contains a finite index torsion-free subgroup $N$ (see section 2.1.1).  Therefore by Lemma \ref{lem:findex} we may
assume that $G$ is torsion free. We decompose $C^*(G)$ as a continuous field over $\widehat{Z(G)}$ as in Theorem
\ref{thm:PackRae}.  Let $\gamma\in \widehat{Z(G)}.$  If $\gamma$ is not faithful on $Z(G)$, then since $G$ is torsion
free we have $h(\textup{ker}(\gamma))>0.$ By our induction hypothesis we then have $\nd(C^*(G,\widetilde{\gamma}))\leq
f(h(G)-1).$ 

Suppose now that $\gamma$ is faithful on $Z(G).$ If $G$ is a 2 step nilpotent group, then $C^*(G,\widetilde{\gamma})$ is
a simple higher dimensional noncommutative torus and therefore an A$\T$ algebra by \cite{Phillips06}.  A$\T$ algebras
have nuclear dimension (decomposition rank in fact) bounded by 1 \cite{Kirchberg04}.

Suppose then that $G$ is nilpotent of class $n\geq3$ and let $Z_{i}(G)$ denote its upper central series.  By
\cite{Malcev49}  (see also \cite[Theorem 1.2]{Jennings55} ), the group $G/Z_{n-1}(G)$ is a free abelian group.  Let
$xZ_{n-1}(G),x_1Z_{n-1}(G),...,x_dZ_{n-1}(G)$ be a free basis for $G/Z_{n-1}(G).$ Let $N$ be the group generated by
$Z_{n-1}(G)$ and $\{ x_1,...,x_d \}.$  Then $N$  is a normal subgroup of $G$ with $h(N)=h(G)-1$ and $G=N\rtimes_\alpha
\Z$ where $\alpha$ is conjugation by $x.$ 

Suppose first that $Z(N)=Z(G).$  Since $h(N)=h(G)-1$, the group C*-algebra $C^*(N)$ has finite nuclear dimension by our
induction hypothesis.  Assuming $Z(N)=Z(G)$ means that $\widetilde{\gamma}$ is a character on $N$,  i.e.
$C^*(N,\widetilde{\gamma})$ is primitive quotient of $C^*(N).$ It then enjoys all of the properties of Theorem
\ref{thm:bigol}.

By Lemma \ref{lem:strongouter}, the action of $\alpha$ on $C^*(N,  \widetilde{\gamma} )$ is strongly outer.  By
\cite[Corollary 4.11]{Matui14}, the crossed product $C^*(N,\widetilde{\gamma})\rtimes_\alpha\Z\cong
C^*(G,\widetilde{\gamma})$  is $\mathcal{Z}$-stable, hence $C^*(G,\widetilde{\gamma})$ has decomposition rank bounded by
3 by Theorem \ref{thm:bigol}.

Suppose now that $Z(G)$ is strictly contained in $Z(N).$ By \cite[Theorem 4.1]{Hirshberg12} and Lemma \ref{lem:rokdim}
we have
\begin{equation*}
  \nd(C^*(G,\widetilde{\gamma}))=\nd(C^*(N,\widetilde{\gamma})\rtimes_\alpha \Z)\leq 8(\nd(C^*(N,\widetilde{\gamma}))+1)\leq 9f(h(N)).
\end{equation*}
Therefore the nuclear dimension of every fiber of $C^*(G)$ is bounded by $9f(h(G)-1)$ and the dimension of its center is
at most $h(G)-1$.  By Theorem \ref{thm:jose}, we have
\begin{equation*}
  \nd(C^*(G))\leq 10h(G)f(h(G)-1)=f(h(G)).
\end{equation*}
\end{proof}

\subsection{Application to the Classification Program}

Combining Theorem \ref{thm:mainthm} with results of Lin and Niu \cite{Lin08} and Matui and Sato \cite{Matui12,Matui14a}
(see especially Corollary 6.2 of \cite{Matui14a}) we display the reach of Elliott's classification program:
\begin{theorem} \label{thm:uctmissing}
  Let $G$ be a finitely generated nilpotent group and $J$ a primitive ideal of $C^*(G).$  If $C^*(G)/J$ satisfies the
  universal coefficient theorem, then $C^*(G)/J$ is classifiable by its ordered K-theory and is isomorphic to an
  approximately subhomogeneous C*-algebra.
\end{theorem}
We do not know if every quotient of a nilpotent group C*-algebra satisfies the UCT, but Rosenberg and Schochet show that
satisfying the UCT is closed under $\Z$-actions which covers most cases of interest;
\begin{theorem} \label{thm:ucttfree}
  Let $G$ be a finitely generated, torsion free nilpotent group and $\pi$ a faithful irreducible representation of $G.$
  Then $C^*(\pi(G))$ is classifiable by its ordered K-theory and is an ASH algebra. 
\end{theorem}
\begin{proof}
By Theorem \ref{thm:uctmissing} we only need to show that $C^*(\pi(G))$ satisfies the UCT.  Let $\tau$ be the unique
character inducing $\textup{ker}(\pi)$ (see Section \ref{sec:nilfacts}).  Since $G$ is torsion free so is $G/Z(G)$
(Section \ref{sec:ngfacts}).  Therefore we have a normal series
\begin{equation*}
  Z(G)=G_0\trianglelefteq G_1 \trianglelefteq \cdots \trianglelefteq G_n=G
\end{equation*}
such that $G_i/G_{i-1}\cong \Z$ for all $i=1,...,n.$ Let $x_1,...,x_n\in G$ be such that $x_iG_{i-1}$ generates
$G_i/G_{i-1}$ for $i=1,...,n.$  For $i=1,...,n$ let $\alpha_i$ be the automorphism of $C^*(\pi(G_{i-1}))$ defined by
conjugation by $x_i.$  Since $C^*(\pi(G))$ is an iterated crossed product by the $\Z$-actions $\alpha_i$, a repeated
application of  \cite[Proposition 2.7]{Rosenberg87} shows that $C^*(\pi(G))$ satisfies the UCT.
\end{proof}
\subsubsection{Unitriangular groups and K-theory}
In light of Theorem \ref{thm:ucttfree} it is natural to wonder about  K-theory calculations for specifc groups and
representations.  Let us announce a little progress in this direction.  Let $d\geq 3$ be an integer.  Consider the group
of upper triangular matrices,
\begin{equation*}
  U_d=\Big\{ A\in GL_d(\Z): A_{ii}=1 \textup{ and }A_{ij}=0 \textup{ for }i>j   \Big\}.
\end{equation*} 
The center $Z(U_d)\cong \Z$ is identified with those elements whose only non-zero, non-diagonal entry can occur in the
$(1,d)$ matrix entry.  $U_d$ is a finitely generated $d-1$-step nilpotent group.  Fix an irrational $\theta$ and
consider the character $\tau_\theta$ on $U_d$ induced from the multiplicative character $n\mapsto e^{2\pi n\theta i}.$
Then $C^*(U_d,\tau_\theta)$ is covered by Theorem \ref{thm:ucttfree}.

In the case of $d=4$, we show in \cite{Eckhardt14a}, together with Craig Kleski that the Elliott invariant of
$C^*(U_4,\tau_\theta)$ is given by $K_0=K_1=\Z^{10}$ with the order on $K_0$ given by those vectors $x\in \Z^{10}$ satisfying $\la x, (1,\theta,\theta^2,0,...,0)\ra>0.$

\section{Questions and Comments}
Very broadly this section contains one question:  Is there a group theoretic characterization of finitely generated
groups whose group C*-algebras have finite nuclear dimension? We parcel this into more manageable chunks.

Since finite decomposition rank implies strong quasidiagonality (see \cite{Kirchberg04}) there are many easy examples of
finitely generated group C*-algebras with infinite decomposition rank \cite{Carrion13}. There are also difficult examples of finitely generated amenable groups with infinite nuclear dimension.

 \begin{theorem}[Giol and Kerr \textup{\cite{Giol10}}] The nuclear dimension of $C^*(\Z\wr\Z)$ is infinite.
 \end{theorem}
 \begin{proof} Giol and Kerr construct several C*-algebras of the form $C(X)\rtimes\Z$ with infinite nuclear dimension.  One notices that some of these algebras are actually quotients of $C^*(\Z\wr\Z)$, forcing the nuclear dimension of $C^*(\Z\wr\Z)$ to be infinite by \cite[Proposition 2.3]{Winter10}.
 \end{proof}

On the other hand if we restrict to
polycyclic groups--given  the finite-dimensional feel of the Hirsch number and the role it played in the present
work--it seems plausible that all of these groups have finite nuclear dimension. Note that in \cite{Eckhardt13b} there are numerous examples of
polycyclic groups that are not strongly quasidiagonal and therefore have infinite decomposition rank. 
\begin{question}
  Are there any polycyclic groups  with infinite nuclear dimension? 
 \end{question}
 \begin{question} Are there any polycyclic, non virtually nilpotent groups with finite decomposition rank or finite nuclear dimension?
\end{question}
 The paper \cite{Eckhardt13b} also provides examples of non virtually nilpotent, polycyclic groups whose group
C*-algebras are strongly quasidiagonal. It seems that these groups could be a good starting point for a general
investigation into nuclear dimension of polycyclic groups.  The difficulty here is that these groups have trivial center
and we therefore do not have a useful continuous field characterization of their group C*-algebras as in Theorem
\ref{thm:PackRae}.

Unfortunately we were unable to extend our results to the virtually nilpotent case, and thus are left with the following
question.
\begin{question}
  Do virtually nilpotent group C*-algebras have finite nuclear dimension?
\end{question}
Finally we have
\begin{question}
  If $G$ is finitely generated and nilpotent, does $C^*(G)$ have finite decomposition rank?
\end{question}
The careful reader will notice that the only part of our proof where we cannot deduce finite decomposition rank is in
the second case of the proof of  Theorem \ref{thm:mainthm}.  There is definitely a need for both cases as there exist
torsion free nilpotent groups $G$ such that whenever $G/N\cong \Z$ with $Z(G)\leq N$, then $Z(N)$ is strictly bigger
than $Z(G)$ (we thank the user YCor on mathoverflow.net for kindly pointing this out to us). In general if a C*-algebra
$A$ has finite decomposition rank and an automorphism has finite Rokhlin dimension, one cannot deduce that the crossed
product has finite decomposition rank (for example, consider $\alpha\otimes \beta$ where $\alpha$ is a shift on a Cantor
space and $\beta$ is an irrational rotation of $\T$).

\bibliographystyle{plain}

\bibliography{mybib}

\end{document}